\documentclass{amsart}
\usepackage{latexsym}
\usepackage{palatino}
\usepackage{amsmath, amssymb,amsmath}
\usepackage{hyperref}

\usepackage[dvipsnames,usenames]{color}
\usepackage[all]{xy}
\newtheorem{dfs}{Definition}[subsection]
\newtheorem{lms}[dfs]{Lemma}
\newtheorem{thms}[dfs]{Theorem}
\newtheorem{props}[dfs]{Proposition}
\newtheorem{qus}[dfs]{Question}
\newtheorem{cors}[dfs]{Corollary}
\newtheorem{rems}[dfs]{Remark}

\newtheorem*{thm*}{Theorem}

\newcommand{\culeq}{\precsim}
\newcommand{\cugeq}{\succsim}
\newcommand{\cueq}{\sim}

\newcommand{\chii}{\raisebox{2pt}{$\chi$}}

\newcommand{\N}{\mathbb{N}}

\renewcommand{\epsilon}{\varepsilon}
\renewcommand{\leq}{\leqslant}
\renewcommand{\geq}{\geqslant}

\title[An algebraic approach to the radius of comparison]
{An algebraic approach to the radius of comparison}

\author[Blackadar]{Bruce Blackadar}
\address{Department of Mathematics and Statistics,  
Ansari Business Building, 601 | Mail Stop 084, 
Reno, NV, USA 89557-0084}
\email{bruceb@unr.edu} 

\author[Robert]{Leonel Robert}
\address{Department of Mathematics and Statistics, York University, 4700 Keele St., Toronto, ON, Canada L4J 3A4}
\email{leonel.robert@gmail.com}

\author[Tikuisis]{Aaron P.\ Tikuisis}
\address{Department of Mathematics, University of Toronto, Room 6290, 40 St. George Street, Toronto, ON, Canada M5S 2E4}
\email{aptikuis@math.toronto.edu}

\author[Toms]{Andrew S.\ Toms}
\address{Department of Mathematics, Purdue University, 150 N. University St., West Lafayette, IN, USA, 47907-2067}
\email{atoms@purdue.edu}

\author[Winter]{Wilhelm Winter}
\address{School of Mathematical Sciences, University of Nottingham, Nottingham NG7 2RD, UK}
\email{wilhelm.winter@nottingham.ac.uk}

\begin{document}

\begin{abstract}
The radius of comparison is an invariant for unital C$^*$-algebras which extends the theory of covering dimension to noncommutative spaces.
We extend its definition to general C$^*$-algebras, and give an algebraic (as opposed to functional-theoretic) reformulation.  This yields new permanence properties for the radius of comparison which strengthen its analogy with covering dimension for commutative spaces.  We then give several applications of these results.  New examples of C$^*$-algebras with finite radius of comparison are given, and the question of when the Cuntz classes of finitely generated Hilbert modules form a hereditary subset of the Cuntz semigroup is addressed.  Most interestingly, perhaps, we treat the question of when a full hereditary subalgebra $B$ of a stable C$^*$-algebra $A$ is itself stable, giving a characterization in terms of the radius of comparison.  We also use the radius of comparison to quantify the least $n$ for which a C$^*$-algebra $D$ without bounded 2-quasitraces or unital quotients has the property that $\mathrm{M}_n(D)$ is stable. \end{abstract}

\maketitle

\section{Introduction}

There are many invariants for C$^*$-algebras which are meant to capture a noncommutative version of covering dimension.  They all recover, or are at least proportional to, the covering dimension of a locally compact Hausdorff space $X$ when applied to the commutative C$^*$-algebra $\mathrm{C}_0(X)$.  Their behaviour for more general C$^*$-algebras, however, can be quite different.  The aim of this note is to expand the theory of one of these invariants:  the radius of comparison.  This invariant may be thought of roughly as being the minimum difference in rank required between positive operators $a$ and $b$ in a C$^*$-algebra $A$ before $b$ can be conjugated to an element arbitrarily close to $a$.  It was introduced in \cite{To-comm} with a view to providing a theory of "moderate dimension growth" for AH algebras, the existence of which was suggested by the first named author in \cite{Bl-matricial}.  It has since proved useful in distinguishing simple C$^*$-algebras having the same $\mathrm{K}$-theory and traces (see \cite{To-comp} and \cite{Giol-Kerr});  \cite{Giol-Kerr} is particularly interesting, as it suggests a connection between the radius of comparison of a crossed product C$^*$-algebra and the mean dimension of the underlying topological dynamical system.

Some of the basic properties of the radius of comparison were established in \cite{To-comm}, but two important questions were left open:  how does the invariant behave with respect to direct limits and quotients?  Here we use an algebraic reformulation of the radius of comparison to answer these questions:  the radius of comparison is lowered by passage to a quotient, and is lower semicontinuous with respect to inductive limits.  The latter result is applied to exhibit a large class of C$^*$-algebras with finite radius of comparison, namely, the ASH algebras of linear or ``flat" dimension growth (\cite{To-flat}).

The Cuntz semigroup encodes a great deal of the structure of a C$^*$-algebra, including its ideal lattice, its (pre)ordered $\mathrm{K}_0$-group (in the stably finite, unital case), and, under some local approximation assumptions, the presence of $\mathcal{Z}$-stability (\cite{w-perfect}).  The radius of comparison is an invariant of the Cuntz semigroup, and it is therefore natural to ask what it can tell us about the structure of a C$^*$-algebra.  We first consider the question of when a $\sigma$-unital full hereditary subalgebra $B$ of a stable C$^*$-algebra $A$ is itself stable.  Two natural necessary conditions, collectively termed property (S) in \cite{OPR}, are that $B$ have no bounded 2-quasitrace and no unital quotients.  When is property (S) sufficient?  R\o rdam proved in \cite{Ror-stab} that the answer is ``not always'', even for simple algebras.  Here we prove that property (S) suffices for stability if the radius of comparison of $A$ relative to the infinite element of its Cuntz semigroup is zero;  if projections are finite in each quotient of $A$, then the sufficiency of property (S) is equivalent to the said vanishing of the radius of comparison.  As a consequence, we prove that the sufficiency of property (S) for the stability of a $\sigma$-unital hereditary subalgebra
$B$ of $A \otimes \mathcal{K}$ is equivalent to the presence of the $\omega$-comparison property of Ortega-Perera-R\o rdam in the Cuntz semigroup of $A$.  As a further consequence, we obtain that a C$^*$-algebra with finite radius of comparison enjoys the Corona Factorization Property of Ng and Kucerovsky.

In proving the insufficiency of property (S) for stability, R\o rdam constructs a simple C$^*$-algebra $A$ which is not stable, but for which $\mathrm{M}_2(A)$ is stable.  This raises a question:  given a C$^*$-algebra $B$, what is the least $n$ such that $\mathrm{M}_n(B)$ is stable?  Of course, one must restrict to algebras with property (S) in order to have any chance of such an $n$ existing. We give a complete answer to the question: when $A$ has property (S), the least $n$ which suffices is exactly one more than the normalised radius of comparison of the unitization of $A$.

Elements of the Cuntz semigroup are now commonly viewed as equivalence classes of countably generated Hilbert modules over a C$^*$-algebra $A$, but in this picture the original definition of Cuntz corresponds to the subsemigroup of classes of finitely generated modules (the latter is typically denoted by $W(A)$).    When is this subsemigroup hereditary?  We prove here that $W(A)$ is a hereditary subset of the Cuntz semigroup whenever the radius of comparison is finite.

The paper is organized as follows:  Section \ref{prelim} introduces the Cuntz semigroup, the category $\mathbf{Cu}$ in which it sits, and some basic facts about its functionals;  Section \ref{rcsection} reformulates the radius of comparison in algebraic terms, and establishes its behaviour with respect to quotients and inductive limits;  Section \ref{applications} establishes the various applications of finite radius of comparison detailed above.

\vspace{3mm}
\noindent
{\it Acknowledgement.}
This paper was born of a collaborative effort at the AIM Workshop ``The Cuntz semigroup", held in November of 2009.  The authors are grateful to AIM and its staff for their support.

A.P.T.\ was supported by an NSERC CGS-D scholarship, A.S.T.\ was supported by NSF grant DMS-0969246, and W.W.\ was supported by EPSRC First Grant  \linebreak EP/G014019/1.

\section{Preliminaries}\label{prelim}

\subsection{Basic notation}

We use $\mathcal{K}$ to denote the algebra of compact operators on a separable infinite-dimensional Hilbert space $\mathcal{H}$.  For a C$^*$-algebra $A$, we use $A_+$ to denote the subset of positive operators.  For $\epsilon>0$ we let $f_\epsilon\colon \mathbb{R} \to \mathbb{R}$ denote the function which is identically zero on $(-\infty,\epsilon]$ and satisfies $f_\epsilon(t)=t-\epsilon$ elsewhere;  for a self-adjoint operator $a$ we  set $(a-\epsilon)_+ := f_\epsilon(a)$.

\subsection{The Cuntz semigroup}\label{WAdef}


Let $A$ be a C$^*$-algebra. Let us consider on
$(A\otimes\mathcal K)_+$ the relation $a\precsim b$ if $v_nbv_n^*\to a$ for some sequence $(v_n)$ in $A\otimes\mathcal K$.
Let us write $a\sim b$ if $a\precsim b$ and $b\precsim a$. In this case we say that $a$ is Cuntz equivalent to $b$.
Let $Cu(A)$ denote the set $(A\otimes \mathcal K)_+/\sim$ of Cuntz equivalence classes. We use $\langle a \rangle$ to denote the class of $a$ in $Cu(A)$.  It is clear that
$\langle a \rangle\leq \langle b \rangle \Leftrightarrow a\precsim b$ defines an order on $Cu(A)$. We also endow $Cu(A)$
with an addition operation by setting $\langle a \rangle+\langle b \rangle:=\langle a'+b' \rangle$, where
$a'$ and $b'$ are orthogonal and Cuntz equivalent to $a$ and $b$ respectively (the choice of $a'$ and $b'$
does not affect the Cuntz class of their sum). The semigroup $W(A)$ is then the subsemigroup of $Cu(A)$ of Cuntz classes
with a representative in $\bigcup_n \mathrm{M}_n(A)_+$.

Alternatively, $Cu(A)$ can be defined to consist of equivalence classes of countably generated Hilbert modules over $A$.  The equivalence relation boils down to isomorphism in the case that $A$ has stable rank one, but is rather more complicated in general.  We do not require the precise definition of this relation in the sequel, and so omit it;  the interested reader may consult \cite{CEI} or \cite{APT} for details.  We note, however, that the identification of these two approaches to $Cu(A)$ is achieved by associating the element $\langle a \rangle$ to the class of the Hilbert module $\overline{a\ell_2(A)}$.  If $X$ is a countably generated Hilbert module over $A$, then we use $[X]$ to denote its Cuntz equivalence class;  with this notation the subsemigroup $W(A)$ is identified with those classes $[X]$ for which $X$ is finitely generated.


\subsection{The category $\mathbf{Cu}$}\label{cu}

The semigroup $Cu(A)$ is an object in a category of ordered Abelian monoids denoted by ${\mathbf{Cu}}$ whose properties we will use heavily.  Before stating them, we require the notion of order-theoretic compact containment.  Let $T$ be a preordered set with $x,y \in T$.  We say that $x$ is compactly contained in $y$---denoted by $x \ll y$---if for any increasing sequence $(y_n)$ in $T$ with supremum $y$, we have $x \leq y_{n_0}$ for some $n_0 \in \mathbb{N}$.  An object $S$ of $\mathbf{Cu}$ enjoys the following properties:
\begin{enumerate}
\item[{\bf P1}] $S$ contains a zero element; \vspace{1mm}
\item[{\bf P2}] the order on $S$ is compatible with addition:  $x_1 + x_2
\leq y_1 + y_2$ whenever $x_i \leq y_i , i \in \{1, 2\}$; \vspace{1mm}
\item[{\bf P3}] every countable upward directed set in $S$ has a supremum; \vspace{1mm}
\item[{\bf P4}] the set $x_\ll = \{y \in S \ |  \ y \ll x\}$ is upward directed with respect to both
$\leq$ and $\ll$, and contains a sequence $(x_n)$ such that
$x_n \ll x_{n+1}$ for every $n \in \mathbb{N}$ and $\sup_n x_n = x$; \vspace{1mm}
\item[{\bf P5}] the operation of passing to the supremum of a countable upward directed set and the relation
$\ll$ are compatible with addition:  if $S_1$ and $S_2$ are countable upward directed sets in $S$, then
$S_1 + S_2$ is upward directed and $\sup (S_1 + S_2) = \sup S_1 + \sup S_2$, and if $x_i \ll y_i$ for $i \in \{1, 2\}$, then
$x_1 + x_2 \ll y_1 + y_2$ .
\end{enumerate}

\noindent
Here we assume further that $0 \leq x$ for any $x \in S$.  This is always the case for $Cu(A)$.
For $S$ and $T$ objects of $\textbf{Cu}$, the map $\phi\colon S\to T$ is a morphism in the category $\mathbf{Cu}$ if
\begin{enumerate}
\item[{\bf M1}] $\phi$ is order preserving; \vspace{1mm}
\item[{\bf M2}] $\phi$ is additive and maps $0$ to $0$; \vspace{1mm}
\item[{\bf M3}] $\phi$ preserves the suprema of increasing sequences; \vspace{1mm}
\item[{\bf M4}] $\phi$ preserves the relation $\ll$. \vspace{1mm}
\end{enumerate}

The category $\mathbf{Cu}$ admits inductive limits, and $Cu(\cdot)$ may be viewed as a functor from C$^*$-algebras into $\mathbf{Cu}$.  A central result of \cite{CEI} is that if $(A_i,\phi_i)$ is an inductive sequence of C$^*$-algebras, then
\[
Cu \left( \lim_{i \to \infty} (A_i,\phi_i) \right) \cong \lim_{i \to \infty} (Cu(A_i), Cu(\phi_i)).
\]
Let $S = \lim_{i \to \infty}(S_i,\phi_i)$ be an inductive limit in the category $\mathbf{Cu}$, with $\phi_{i,j}\colon S_i \to S_j$ and
$\phi_{i,\infty}\colon S_i \to S$ the canonical maps.  We have the following
two properties (established in \cite{CEI}):
\begin{enumerate}
\item[{\bf L1}] each $x \in S$ is the supremum of an increasing sequence $(x_i)$ belonging to $\bigcup_{i=1}^\infty \phi_{i,\infty}(S_i)$
and such that $x_i\ll x_{i+1}$ for all $i$; \vspace{1mm}
\item[{\bf L2}] If $x,y\in S_i$ and $\phi_{i,\infty}(x)\leq \phi_{i,\infty}(y)$,
then for all $x'\ll x$ there is $n$ such that \mbox{$\phi_{i,n}(x')\leq \phi_{i,n}(y)$}.
\end{enumerate}

For $e\in S$ we denote by  $\infty \cdot e$ the supremum $\sup_{n\geq 1} ne$. We say that $e$ is
full if $\infty \cdot e$ is the largest element of $S$. We say that $e$ is compact if $e\ll e$.  If a sequence $(x_i)$ in $S$ satisfies $x_i \ll x_{i+1}$ for every $i$, then we say that the sequence is rapidly increasing.

\subsection{Functionals and $\mathbf{Cu}$}\label{functionals}

Let $S$ be a semigroup in the category $\mathbf{Cu}$. A functional on $S$ is a map $\lambda\colon S\to [0,\infty]$
that is additive, order preserving, preserves suprema of increasing sequences and satisfies $\lambda(0)=0$.
We use $F(S)$ to denote the functionals on $S$.
We will make use of the following lemma, established in \cite{ERS}.
\begin{lms}\label{makelsc}
If $S$ is in the category $\mathbf{Cu}$ and $\lambda\colon S\to [0,\infty]$ is additive, order preserving,
and maps $0$ to $0$, then $\tilde\lambda(x):=\sup_{x'\ll x} \lambda(x')$ defines a functional on $S$.
\end{lms}
For a C$^*$-algebra $A$, the functionals on $Cu(A)$ admit a description in terms of 2-quasitraces.  Recall that a lower semicontinuous extended 2-quasitrace on $A$ is a lower semicontinuous map $\tau\colon (A \otimes \mathcal{K})_+ \to [0,\infty]$ which vanishes at $0$, satisfies the trace identity, and is linear on pairs of positive elements that commute.  The set of all such quasitraces is denoted by $\mathrm{QT}_2(A)$.  Given $\tau \in \mathrm{QT}_2$ we define a map
$d_\tau\colon  Cu(A) \to [0,\infty]$ by the following formula:
\[
d_\tau(\langle a \rangle) := \lim_{n \to \infty} \tau(a^{1/n}).
\]
By Proposition 4.2 of \cite{ERS} that the association $\tau \mapsto d_\tau$ defines a bijection between $\mathrm{QT}_2(A)$ and $F(Cu(A))$, extending the work of Blackadar and Handelman in \cite{BH}.  In particular, $d_\tau(\langle a \rangle)$ is independent of the representative $a$ of $\langle a \rangle$.

\section{The radius of comparison}\label{rcsection}

\subsection{Original definition}The radius of comparison was originally introduced in \cite{To-flat} as an invariant for unital C$^*$-algebras.  Let $A$ be a unital C$^*$-algebra, and let $\mathrm{QT}_2^1(A)$ denote the set of normalised 2-quasitraces on $A$.  The radius of comparison of $A$, denoted by $\mathrm{rc}(A)$, is the infimum of the set of real numbers $r >0$ with the property that $a,b \in \bigsqcup_{n=1}^\infty \mathrm{M}_n(A)$ satisfy $a \precsim b$ whenever
\begin{equation}\label{originaldef}
d_\tau(\langle a \rangle) + r < d_\tau(\langle b \rangle), \ \forall \tau \in \mathrm{QT}_2^1(A).
\end{equation}
By the results of Subsection \ref{functionals}, this is equivalent to the demand that $x,y \in Cu(A)$ satisfy $x \leq y$ whenever
\[ \lambda(x) + r < \lambda(y), \]
for all $\lambda \in F(Cu(A))$ which are normalised in the sense that $\lambda(\langle 1_A\rangle)=1$.

The motivation for this definition comes from the stability properties of topological vector bundles.  It is well known that if two complex vector bundles over a compact Hausdorff space $X$---equivalently, two finitely generated projective modules over $\mathrm{C}(X)$---differ in rank by at least half the covering dimension of $X$, then the bundle (or module) of larger rank dominates the smaller one up to isomorphism \cite[Proposition 1]{dupreII}.  The smallest rank gap required in order to have this kind of comparison up to isomorphism does not, however, determine the dimension of $X$  (no gap is required for contractible $X$).  To recover the dimension of $X$ from comparability properties of modules, one must pass to the countably generated Hilbert modules over $X$.  Equivalence classes of these modules can be identified with $Cu(A)$, so that $\mathrm{rc}(A)$ really gives the minimum rank gap required between Hilbert modules over $A$ in order to guarantee that the larger module dominates the smaller one.  The notion of domination employed here is Cuntz comparison, as formulated in Subsection \ref{WAdef} and translated to the realm of Hilbert modules via \cite{CEI}.  In the case that $A \cong \mathrm{C}(X)$ for a CW-complex $X$, we have $\mathrm{rc}(A) \approx \mathrm{dim}(X)/2$.  This justifies viewing the radius of comparison as a sort of noncommutative dimension, an idea that will be reinforced in the sequel.

Our aim for the remainder of this section is to extend the definition of the radius of comparison to nonunital algebras, and to reformulate it in more algebraic terms.  This reformulation will be used to establish new properties of the radius of comparison which strengthen the analogy with covering dimension for commutative spaces.

\subsection{Compact normalization}
One may take a somewhat more sophisticated view of the number $r$ in the definition of $\mathrm{rc}(A)$, namely, that it is $r \cdot d_\tau(1_A) = r \cdot d_\tau(\langle 1_A \rangle)$.  As a first step toward a general radius of comparison, we consider replacing $\langle 1_A \rangle$ with an arbitrary full compact element $e$ of $Cu(A)$.  Let $S$ be an object of $\mathbf{Cu}$, and let $e \in S$.  Consider the following two properties of a number $r > 0$:
\begin{enumerate}
\item[{\bf R1}] if $x,y\in S$ are such that
\[
\lambda(x)+r\lambda(e)\leq \lambda(y)
\]
for all $\lambda\in F(S)$, then $x\leq y$; \vspace{1mm}
\item[{\bf R2}] if  $x,y\in S$ are such that
\[(n+1)x+me\leq ny,\]
for some $n,m$ which satisfy $\frac{m}{n}>r$, then $x\leq y$.
\end{enumerate}

\begin{props}\label{radius12}  Let $S$ be an object in $\mathbf{Cu}$ with $e \in S$ full and compact.
If $r$ satisfies {\bf R1} then it satisfies {\bf R2}. If $r$ satisfies {\bf R2} then $r+\epsilon$ satisfies {\bf R1}
for any $\epsilon>0$.
\end{props}
\begin{proof}
The implication  ``$r$ satisfies  {\bf R1}'' $\Rightarrow$ ``$r$ satisfies {\bf R2}'' is clear.

Let $r$ be a number that satisfies {\bf R2} and let us show that $r+\epsilon$ satisfies {\bf R1} for $\epsilon>0$.
Suppose that $x,y\in S$ are such that
\begin{align*}\label{strictlambda}
\lambda(x)+(r+\epsilon)\lambda(e)\leq\lambda(y)
\end{align*} for all $\lambda\in F(S)$.
The map $\lambda_y\colon S\to [0,\infty]$
defined by $\lambda_y(z)=0$ if $z\leq \infty \cdot y$ and $\lambda_y(z)=\infty$ otherwise, is
a functional. The above inequality implies that $\lambda_y\equiv 0$. That is, $y$ is full.

Choose $m,n\in \N$ such that  $r<\frac{m}{n}<r+\epsilon$.
Then $\lambda(nx+me)< \lambda(ny)$ for all $\lambda\in F(S)$ such that $\lambda(e)\neq 0$.
Let $x'\ll x$ and let $\lambda\colon S\to [0,\infty]$ be additive, order preserving
(though not necessarily a functional), and satisfy $\lambda(e)\neq 0$. Let $\tilde\lambda$ be the
functional obtained from $\lambda$ as in Lemma \ref{makelsc}.
Then
\[
\lambda(nx'+me)\leq \tilde\lambda(nx+me)<\tilde \lambda(ny)\leq \lambda(ny).
\]
That is, $\lambda(nx'+me)< \lambda(ny)$ for any $\lambda\colon S\to [0,\infty]$
that is additive, order preserving, and satisfies $\lambda(e)\neq 0$.
Notice that  $x',e\leq ky$ for some $k\in\mathbb N$ and that the
inequality $\lambda(nx'+me)<\lambda(ny)$ holds for all $\lambda$ such that $\lambda(y)=1$.
By \cite[Lemma 2.8]{Black-Ror} applied with $y$ as the order unit, we conclude that
\[Nnx'+Nme+z+y\leq Nny+z\]
for some $N\in \N$ and $z\in S$ such that $z\le ky$ for some $k\in \N$. By \cite[Lemma 2.3]{Black-Ror} (again
with $y$ as the order unit), we obtain
\[
N_1Nnx'+N_1Nme+y\leq N_1Nny
\]
for some $N_1>0$. Let $N_2>0$ be such that $x'\leq N_2y$. Then
\[
(N_2N_1Nn+1)x'+N_2N_1Nme\leq N_2N_1Nny.
\]
Notice now that $\frac{N_2N_1m}{N_2N_1Nn}=\frac{m}{n}>r$.
Thus, $x'\leq y$ by {\bf R2}.
Since $x$ is the supremum of $x'$ with $x'\ll x$ we conclude that
$x\leq y$. Thus, $r+\epsilon$ satisfies {\bf R1}.
\end{proof}

\begin{dfs}\label{compactrc}
Let $S$ be an object of $\mathbf{Cu}$, and let $e \in S$ be full and compact.  We define the radius of comparison of $(S,e)$ to be the infimum of the numbers satisfying {\bf R1} (equivalently, the infimum of the numbers satisfying {\bf R2}) above.
\end{dfs}

For a unital C$^*$-algebra  $A$ we use $r_A$ to denote the radius of comparison of the pair $(Cu(A),\langle 1_A \rangle)$.
We can see from {\bf R1} and the results of Subsection \ref{functionals} that $r_A$ is defined similarly to $\mathrm{rc}(A)$, the only difference being that in {\bf R1}, we allow functionals on $Cu(A)$ which take the value $\infty$ at the unit.
From this it is clear that $r_A \leq \mathrm{rc}(A)$.
In the case that the C$^*$-algebra $A$ is sufficiently finite, the two notions agree:

\begin{props}\label{oldradiusequalsnew}
Let $A$ be a unital C$^*$-algebra all of whose quotients are stably finite.
Then $\mathrm{rc}(A) = r_A$.
\end{props}

\begin{proof}
We need only show that $\mathrm{rc}(A) \leq r_A$.
Let $r$ be a number satisfying {\bf R1} for $(S,e)=(Cu(A),\langle 1_A \rangle)$.
Suppose that $\langle a \rangle,\langle b \rangle \in Cu(A)$ satisfy
\[ \lambda(\langle a \rangle) + r < \lambda(\langle b \rangle) \]
for all $\lambda \in F(Cu(A))$ for which $\lambda(\langle 1_A \rangle)=1$.
Hence,
\begin{equation} \lambda(\langle a \rangle) + r\lambda(\langle 1_A \rangle) < \lambda(\langle b \rangle) \label{oldradiusequalsnew-condition} \end{equation}
for all $\lambda \in F(Cu(A))$ for which $\lambda(\langle 1_A \rangle)<\infty$.

If $\lambda \in F(Cu(A))$ satisfies $\lambda(\langle 1_A \rangle) = \infty$ then to show $\lambda(\langle b \rangle)=\infty$ we must show that $y$ is full in $A \otimes \mathcal{K}$.
Suppose, for a contradiction, that $y$ is not full, so that the ideal $I$ generated by $y$ is not all of $A \otimes \mathcal{K}$.
Since $(A \otimes \mathcal{K})/I$ is finite, we can define $\overline{\lambda} \in F(Cu((A\otimes \mathcal{K})/I))$ that is non-zero and satisfies $\overline{\lambda}(\langle 1_A + I\rangle)=1$.
Then $\overline{\lambda}$ induces a functional on $Cu(A)$ which sends $\langle 1_A \rangle$ to $1$ but $\langle b \rangle$ to $0$, contradicting \eqref{oldradiusequalsnew-condition}.
Hence, $b$ is full and so \eqref{oldradiusequalsnew-condition} holds for all $\lambda \in F(Cu(A))$.
By {\bf R1}, we have $a \culeq b$, as required.
\end{proof}

In contrast to the agreement between $\mathrm{rc}(A)$ and $r_A$ in the finite case, if $A$ is a purely infinite simple C*-algebra, $\mathrm{rc}(A)=\infty$ but $r_A=0$ (in this case $Cu(A) = \{0,\infty\}$ so that the only non-trivial functional is the one taking $\infty$ to $\infty$).

\begin{props}\label{compactproperties}
Let $A$ be a unital C$^*$-algebra.

\begin{itemize}
\item[(i)] For any closed two-sided ideal $I$ of $A$ we have $r_{A/I}\leq r_A$.
\item[(ii)] If $A$ is simple  then $r_A=0$ is equivalent to almost unperforation, i.e., $(k+1)x\leq ky$ for some $k$ implies $x\leq y$.
\item[(iii)] If $A=\varinjlim A_i$, where the homomorphisms of the inductive limit are unital, then \[r_A\leq \liminf r_{A_i}.\]
\end{itemize}
\end{props}

\begin{proof}
{(i).} Suppose that $x,y\in Cu(A/I)$ satisfy {\bf R2} for $e=\langle 1\rangle $
and $r=r_A$.
Let $\tilde x$ and $\tilde y$
be lifts of $x$ and $y$ in $Cu(A)$.
Then the inequality $(n+1)x+m\langle 1\rangle\leq ny$ lifts to
\[
(n+1)\tilde x+m\langle 1\rangle \leq n\tilde y+z\leq n(\tilde y +z).
\]
for some $z\in Cu(I)\subseteq Cu(A)$. Thus, $\tilde x\leq \tilde y+z$ and passing to the quotient we get $x\leq y$. This shows that $r_{A/I}\leq r_A$.

\vspace{2mm}
\noindent
{(ii).} It is clear, by the characterization {\bf R2} of the radius of comparison,
that almost unperforation implies $r_A=0$.
Suppose that $r_A=0$. Let $x,y\in Cu(A)$ be such that $(n+1)x\leq ny$ for some $n\in \N$.
Since $A$ is simple, every element of $Cu(A)$ is full (this is equivalent to simplicity).
Hence $\langle 1_A \rangle \leq \infty\cdot  x$, and so $\langle 1_A \rangle \leq kx$
for some $k$. We have  $knx+ \langle1_A \rangle \leq k(n+1)x\leq kny$. Since $r_A=0$, we get that $x\leq y$,
as desired.

\vspace{2mm}
\noindent
{(iii).} It is enough to show that $r_A\leq \sup_i r_{A_i}$ for all $i$
(passing to subsequences of the inductive limit
we  get the $\liminf$). Suppose we have $x$ and $y$ in $Cu(A)$ such that
$(n+1)x+m\langle 1\rangle\leq ny$ for $\frac{m}{n}>\sup_i r_{A_i}$. Let us show that $x\leq y$.
By {\bf L1} (see Subsection \ref{cu}) it suffices to assume that $x$ and $y$ come from finite stages.
Suppose that $x,y\in Cu(A_i)$ are such $(n+1)\phi_{i,\infty}(x)+m\langle 1_A \rangle \leq n\phi_{i,\infty}(y)$. Let $x'\ll x$.
By {\bf L2}, $(n+1)\phi_{i,j}(x')+m\langle 1_A \rangle \leq n\phi_{i,j}(y)$ for some $j\geq i$. Since $\frac{m}{n}>r_{A_j}$,
$\phi_{i,j}(x')\leq \phi_{i,j}(y)$ and so $\phi_{i,\infty}(x')\leq \phi_{i,\infty}(y)$. Since
$x'\ll x$ is arbitrary, we conclude that $\phi_{i,\infty}(x)\leq \phi_{i,\infty}(y)$.
\end{proof}

\begin{rems} {\rm Proposition \ref{compactproperties} shows that the radius of comparison enjoys some properties analogous to those of the covering dimension of a space (or commutative C$^*$-algebra).
From the general theory of covering dimension one has that a direct system of commutative C$^*$-algebras $(\mathrm{C}(X_i), \phi_i)$ with limit $\mathrm{C}(X)$ satisfies
\[
\mathrm{dim}(X) \leq  \liminf_i \mathrm{dim}(X_i).
\]
Proposition \ref{compactproperties} (iii) uses the radius of comparison to extend this property to unital C$^*$-algebra direct limits.  Similarly, Proposition \ref{compactproperties} (i) can be seen as a C$^*$-algebra extension of the fact that the covering dimension of a closed subset of a compact Hausdorff space is less than or equal to the dimension of the original space.
 }
 \end{rems}

\subsection{General normalization}

We now extend the definition of the radius of comparison to pointed objects $(S,e)$ in $\mathbf{Cu}$ for which $e$ is full but not necessarily compact.  In this case the infimum of the numbers $r>0$ satisfying {\bf R1} above still gives a reasonable definition, but in order to have an equivalent algebraic definition like {\bf R2} we have to make an adjustment.  Consider, then, the following property for a number $r >0$:
\begin{enumerate}
\item[{\bf R2'}]  if  $x,y\in S$ are such that   for all $x'\ll x$ and $e'\ll e$ there
are $n,m$,  with $\frac{m}{n}>r$, such that
\[(n+1)x'+me'\leq ny,\]
then $x\leq y$.
\end{enumerate}

\begin{props}\label{radiusgeneral}  Let $S$ be an object in $\mathbf{Cu}$ with $e \in S$ full.
If $r$ satisfies {\bf R1} then it satisfies {\bf R2'}. If $r$ satisfies {\bf R2'} then $r+\epsilon$ satisfies {\bf R1} for any $\epsilon>0$.
\end{props}

\begin{proof}
Let  $r$ be a number satisfying  {\bf R1}. Suppose that $x$ and $y$ are as in  {\bf R2'}.
Then for $\lambda\in F(S)$, $x' \ll x$ and $e' \ll e$ we have
$\lambda(x')+r\lambda(e')\leq \lambda(y)$. Taking supremum over all $x'\ll x$ and $e'\ll e$ we conclude that
$\lambda(x)+r\lambda(e)\leq\lambda(y)$. By {\bf R1} this implies that $x\leq y$, and so $r$ satisfies
{\bf R2'}.

Let $r$ be a number that satisfies {\bf R2'}, and let $\epsilon>0$ be given.
Suppose that
\begin{align}\label{strictlambda}
\lambda(x)+(r+\epsilon)\lambda(e)\leq\lambda(y)
\end{align} for some $x,y\in S$
and for all $\lambda\in F(S)$. By restricting to $\lambda$ such that  $\lambda(e)\neq 0$, and making $\epsilon$
smaller, we may assume
that the inequality is strict. Choose $m,n\in \N$ such that  $r<\frac{m}{n}<r+\epsilon$.
Then $\lambda(nx+me)<\lambda(ny)$ for all $\lambda\in F(S)$ such that $\lambda(e)\neq 0$.
Let $x'\ll x$ and $e'\ll e$, and let $\gamma\colon S \to [0,\infty]$ be additive, order preserving, and satisfy $\gamma(0)=0$ and $\gamma(e) \neq 0$.  It follows that
\[
\gamma(nx'+me') \leq \tilde\gamma(nx' + me') < \tilde\gamma(ny) \leq \gamma(ny),
\]
where $\tilde\gamma$ is defined as in Lemma \ref{makelsc}.
Notice that $y$ is an order unit for $x'$ and $e'$, and that the
inequality $\gamma(nx'+me')<\gamma(ny)$ holds in particular for those $\gamma$ which satisfy $\gamma(y)=1$.
Now by \cite[Lemma 2.8]{Black-Ror} we have
\[Nnx'+Nme'+z+y\leq Nny+z\]
for some $N\in \N$ and $z\in S$ such that $z\le ky$ for some $k\in \N$. By \cite[Lemma 2.3]{Black-Ror} (with $y$ the order unit), we obtain
\[
N_1Nnx'+N_1Nme'+y\leq N_1Nny
\]
for some $N_1>0$. Let  $N_2$ be such that $x'\leq N_2y$. Then
\[
(N_2N_1Nn+1)x'+N_1Nme'\leq N_2N_1Nny
\]
Notice now that $\frac{N_2N_1Nm}{N_2N_1Nn}=\frac{m}{n}>r$. Since a similar inequality
may be obtained for any $x'\ll x$ and $e'\ll e$, we conclude by {\bf R2'} that
$x\leq y$. Thus, $r+\epsilon$ satisfies {\bf R1}.
\end{proof}

 \begin{dfs}\label{generalrc}
Let $A$ be a C$^*$-algebra with $a\in A\otimes \mathcal K$ full and positive.  The radius of comparison of $A$ relative to $a$, denoted by $r_{A,a}$, is the infimum of the numbers $r>0$ satisfying {\bf R1} (or {\bf R2'}) with respect to $(Cu(A),\langle a \rangle)$.
\end{dfs}

It is straightforward to see that Definition \ref{generalrc} coincides with Definition \ref{compactrc} when $\langle a \rangle$ is compact.  We chose to treat the compact case separately both because {\bf R2} is rather cleaner than {\bf R2'} and because the radius of comparison relative to a compact element has stronger permanence properties.  For the radius of comparison relative to a general full positive element $a$, we can nevertheless prove parts (i) and (ii) of Proposition \ref{compactproperties}.

\begin{props}\label{generalproperties} Let $A$ be a C$^*$-algebra, with $a \in A \otimes \mathcal{K}$ full and positive.
\begin{enumerate}
\item[(i)] For a closed two-sided $\sigma$-unital ideal $I$ of $A$ we have $r_{A/I,\pi(a)}\leq r_{A,a}$.
\item[(ii)] If $A$ is simple and $\langle a \rangle \ll \infty$ (e.g., $a=(b-\epsilon)_+$ for some $b$ and $\epsilon>0$), then
$r_{A,a}=0$ if and only if $Cu(A)$ is almost unperforated.
\end{enumerate}
\end{props}

\begin{proof}
{(i).} Let $\pi^*\colon Cu(A) \to Cu(A/I)$
denote the map induced by the quotient map $\pi\colon A\to A/I$.
Let $x,y\in Cu(A/I)$ satisfy {\bf R2'} for $e=\langle a \rangle$ and $r=r_{A,a}$. Choose $\tilde x$ and $\tilde y$,
lifts of $x$ and $y$ in $Cu(A)$, and let $x',e'\in Cu(A)$ be such that $\tilde x'\ll \tilde x$ and $e'\ll \langle a \rangle$.

We have $\pi^*(x')\ll x$ and $\pi^*(e')\ll \langle \pi(a) \rangle$, and so there are
$m,n\in \N$ such that $\frac{m}{n}>r$ and $(n+1)x'+me'\leq ny$. This inequality
lifts to
\[(n+1)\tilde x'+me'\leq n\tilde y+z_I\leq n(\tilde y+z_I),\]
where $z_I$ is the largest element of $Cu(I)$.
Since this holds for all $\tilde x'$ and $e'$, we conclude
that $\tilde x\leq \tilde y+z_I$. Passing to the quotient, we get $x\leq y$. This shows that $r_{A/I, \pi(a)}\leq r_{A,a}$.

\vspace{2mm}
\noindent
{(ii).} Suppose that $(n+1)x\leq ny$ for some $n\in \N$ and some $x,y\in Cu(A)$.
Since $A$ is simple, every element of $Cu(A)$ is full, and so $\langle a \rangle \leq kx$
for some $k$. We then have  $knx+\langle a \rangle \leq k(n+1)x\leq kny$, so that
\[
\lambda(x) + \frac 1 {kn}\lambda(\langle a \rangle) \leq \lambda(y), \ \forall \lambda \in F(Cu(A)).
\]
 Since $r_{A,a}=0$ we have $x\leq y$, so that $Cu(A)$ is almost unperforated.

 Now suppose that $Cu(A)$ is almost unperforated, and that $x,y \in Cu(A)$ satisfy
 \[
 \lambda(x) + r \lambda(\langle a \rangle) \leq \lambda(y)
 \]
 for each $\lambda \in F(Cu(A))$ and some $r>0$.  Shrinking $r$ slightly, we may assume that the inequality is strict for $\lambda\neq 0$.  Proceeding as in the proof of Proposition \ref{radiusgeneral}, we see that for any $x' \ll x$ we have $\gamma(x') < \gamma(y)$ for each $\gamma\colon Cu(A) \to [0,\infty]$ that is additive, order preserving, and satisfies $\gamma(0)=0$.  It now follows from \cite[Proposition 3.2]{Ror-Zabs} that $x' \leq y$, so that $x \leq y$ by taking a supremum.  This shows that $A$ satisfies {\bf R1} relative to $\langle a \rangle$ for arbitrarily small values of $r$, and so $r_{A,a}=0$.
\end{proof}

In the next section we consider the radius of comparison with respect to the largest element
of $Cu(A)$. Suppose that $A$ is $\sigma$-unital and let $a$ be a strictly positive element
of $A\otimes\mathcal K$.  Set $\infty = \langle a \rangle$, which is the maximum element in $Cu(A)$.
We then have $r_{A,\infty}<\infty\Leftrightarrow r_{A,\infty}=0$, and in turn this is equivalent to
\begin{align}\label{rinfty}
\lambda(y)=\infty \hbox{ for all non-zero }\lambda\in F(Cu(A))\Leftrightarrow y=\infty.
\end{align}
We shall see that this property is a strengthening of the Corona Factorization Property.
Notice also that if  $r_{A,a}<\infty$ for some full positive $a \in A \otimes \mathcal{K}$, then $r_{A,\infty}=0$.

\section{Applications to C$^*$-algebras with finite radius of comparison}\label{applications}

C$^*$-algebras with finite nonzero radius of comparison are pathological from a certain point of view:  they are not classifiable up to isomorphism via $\mathrm{K}$-theoretic invariants.  Theorem 5.11 of \cite{To-comp} exhibits, for each $r \in \mathbb{R}^+ \backslash \{0\}$, a unital simple C$^*$-algebra $A_r$ with radius of comparison $r$ such that the Elliott invariants of $A_r$ and $A_s$ are identical for any $r,s$.  We shall nevertheless prove here that C$^*$-algebras with finite radius of comparison do enjoy some good properties, and that the radius of comparison can even be used to characterise interesting structural properties of C$^*$-algebras.

\subsection{New examples}  Recall that a recursive subhomogeneous (RSH) algebra is an iterated pullback of the form
\begin{equation}\label{rshdef}
\left[ \cdots \left[ \left[ \mathrm{M}_{n_1}(\mathrm{C}(X_1)) \oplus_{C_1} \mathrm{M}_{n_2}(\mathrm{C}(X_2)) \right] \oplus_{C_2} \mathrm{M}_{n_3}(\mathrm{C}(X_3)) \right] \oplus_{C_3} \cdots \right] \oplus_{C_{l-1}} \mathrm{M}_{n_l}(\mathrm{C}(X_l))
\end{equation}
where each $X_i$ is a compact metric space and each $C_i$ has the form $\mathrm{M}_{n_{i+1}}(\mathrm{C}(Y_i))$ with $Y_i \subseteq X_{i+1}$ closed (see \cite{Phil-rsh}).  A unital separable ASH algebra is always an inductive limit of RSH algebras (\cite{nw}).
It was shown in \cite{To-comp} that an RSH algebra $A$ with decomposition as in (\ref{rshdef}) satisfies
\[
r_A \leq \min_{1 \leq i \leq l}  \frac{\mathrm{\dim}(X_i)}{2n_i}.
\]
In fact, slightly more precise information can be obtained, based on the fact that for $B=\mathrm{M}_n(\mathrm{C}_0(X))$ and
$b \in B_+$ strictly positive we have
\[
r_{A,a}\leq \left\{
\begin{array}{ll}
\frac{\dim X-2}{2n} & \hbox{ if $\dim X$ is even,}\\
\frac{\dim X-3}{2n} & \hbox{ if $\dim X$ is odd.}
\end{array}
\right.
\]
At any rate, if one has a unital inductive sequence $(A_i,\phi_i)$ of RSH algebras with the property that $\liminf r_{A_i} < \infty$, then the limit algebra $A$ satisfies $r_A < \infty$ by Proposition \ref{compactproperties}.  In particular, the linear (flat) dimension growth AH algebras considered in \cite{To-flat} have finite radius of comparison.

\begin{rems}{\rm
The Corona Factorization Property (CFP) was introduced by Kucer\-ovsky and Ng in \cite{Kuc-Ng}, and is related to the study of absorbing extensions.  It has several equivalent formulations.  An attractive one for $\sigma$-unital C$^*$-algebras is the following:
a $\sigma$-unital C$^*$-algebra $A$ has the CFP if whenever $B\subseteq A\otimes \mathcal K$ is $\sigma$-unital, full, hereditary, and satisfies that $M_n(B)$ is stable for some $n$, then $B$ is stable.
If $A \cong \mathrm{C}_0(X) \otimes \mathcal{K}$, then $A$ has the CFP whenever $X$ is finite-dimensional (and somewhat more generally);  the presence of the CFP in this case is a manifestation of finite-dimensionality.  Finite radius of comparison is, to some degree, a noncommutative generalization of finite-dimensionality for spaces, and so one might expect it to be related to the CFP, too.  This is indeed the case, as it follows from Theorem \ref{corona} below that finite radius of comparison implies the  CFP.  This phenomenon was already observed implicitly for certain unique trace C$^*$-algebras in \cite{Kuc-Ng-2}.}
\end{rems}

\subsection{Radius of comparison relative to $\infty$ and stability of hereditary subalgebras}
How can one characterise stable C$^*$-algebras?  Straightforward necessary conditions for stability are the absence of nonzero bounded 2-quasitraces and unital quotients, conditions collectively termed property (S) in \cite{OPR}.  These conditions, however, are not sufficient in general.  Examples of C$^*$-algebras with the property (S) that are not stable can be found
among the hereditary subalgebras of $\mathrm{C}([0,1]^\N)\otimes \mathcal K$. R\o rdam's  example of a simple C$^*$-algebra that is not stable but becomes
stable after tensoring with $\mathrm{M}_2$ also has the property (S) without being stable.
Here we will prove that if $Cu(A)$ has  finite radius
of comparison then the property (S) implies stability for the full hereditary subalgebras of
$A\otimes \mathcal K$, and so in particular for $A$. This generalizes \cite[Theorem 3.6]{HRW}, which covers the case of strict comparison.

In \cite[Proposition 4.9]{OPR}, it is shown that if $A$ is unital and a certain comparability condition is verified on $Cu(A)$ (named weak $\omega$-comparison in \cite{OPR}) then the property (S) implies stability for the $\sigma$-unital full hereditary subalgebras of $A\otimes\mathcal K$. In Theorem \ref{corona} below we remove the requirement that $A$ be unital and give a condition on $Cu(A)$ equivalent to the stability of every $\sigma$-unital full hereditary subalgebra of $A\otimes\mathcal K$ that has the property (S).

For elements $x$ and $y$ of an ordered semigroup, we write $x\leq_s y$
if \mbox{$(k+1)x\leq ky$} for some $k\in \N$.

\begin{thms}\label{corona}
Let $A$ be a C$^*$-algebra that contains a full element. Consider the following propositions:

\begin{enumerate}
\item[(i)] If  $(x_i)_{i=1}^\infty$ and $(y_i)_{i=1}^\infty$ are sequences  in $Cu(A)$
such that $x_{i-1}\leq x_i\leq_s y_i$ for all $i$ and $\sup_i x_i$ is a full element of $Cu(A)$, then $\sum_{i=1}^\infty y_i=\infty$.
\item[(ii)] If a $\sigma$-unital full hereditary subalgebra of $A\otimes \mathcal K$  has
the property (S), then it is stable.
\item[(iii)] $r_{A,\infty}=0$.
\end{enumerate}

Then (i) and (ii) are equivalent and are implied by (iii). If in every quotient of $A\otimes \mathcal K$ projections are finite then
(iii) is equivalent to (i) and (ii).
\end{thms}

\begin{rems}{\rm
By Proposition \ref{stable-unitization} and Remark \ref{bad-ideal} below, we see that there are C$^*$-algebras $A$ which satisfy $r_{A,\infty}=0$ but such that (ii) does not hold for ideals of $A$.}
\end{rems}

Before proving Theorem \ref{corona} we need the following lemma, which is a slight refinement of \cite[Lemma 4.5]{OPR}.
In the statement of this lemma $F(B)$ denotes the set $\{c\in B_+\mid ec=c\hbox{ for some }e\in B_+\}$.

\begin{lms}\label{s-hjelmborg-rordam}
Let $B$ be a $\sigma$-unital C$^*$-algebra with the property (S) and let $b\in B_+$ be strictly positive. Then for every $a\in F(B)$ and $\epsilon>0$
there is $c\in B_+$ such that $ac=0$, $a+c\in F(B)$, $\langle a\rangle \leq_s \langle c\rangle $, and $\langle (b-\epsilon)_+\rangle\leq_s \langle c\rangle$.
\end{lms}

\begin{proof}
Let $e,f\in B_+$ be such that $fa=a$ and $ef=f$. Then the element
\[
\tilde b:=e+(1-e)b(1-e)\geq \frac b 2
\]
is strictly positive in $B$. Thus, there is $\delta>0$
such that $\langle (b-\epsilon)_+\rangle \leq \langle(\tilde b-\delta)_+\rangle$.
Since  $(f-\delta)_+a=(1-\delta)a$ and $(f-\delta)_+\leq (\tilde b-\delta)_+$,
we have $a\in \overline{(\tilde b-\delta)_+B(\tilde b-\delta)_+}$. Now by \cite[Lemma 4.5]{OPR} there exists $c\in B_+$ such that
$(\tilde b-\delta)_+c=0$, $(\tilde b-\delta)_+ + c\in F(B)$, and
$\langle (\tilde b-\delta)_+\rangle \leq_s \langle c\rangle$. This is the desired element $c$.
\end{proof}

\noindent
We can now prove Theorem \ref{corona}.  Recall that positive elements $a,b$ in a C$^*$-algebra $A$ are said to be Murray-von Neumann equivalent if there exists $x \in A$ such that $x^*x=a$ and $xx^*=b$.

\begin{proof}
(i)$\Rightarrow$(ii). We follow the same line of reasoning used in  the proof of  \cite[Proposition 4.8]{OPR}.
Let $B$ be a $\sigma$-unital full hereditary subalgebra of $A\otimes \mathcal K$ with the property (S). Let $b\in B_+$ be a strictly positive element of $B$.
In order to show that $B$ is stable, it suffices, by the Hjelmborg-R\o rdam stability criterion (see \cite[Theorem 2.1 and Proposition 2.2]{hjelmborg-rordam}),
to show that for every $\epsilon>0$ there is $c\in B_+$ such that $(b-\epsilon)_+$ is Murray-von Neumann equivalent to  $c$ and $bc=0$. Starting with
the positive element $(b-\epsilon)_+$, and repeatedly applying
Lemma \ref{s-hjelmborg-rordam}, we find a sequence of elements $b_i\in B_+$, $i=1,2,\dots$,  such that $(b-\epsilon)_+,b_1,b_2,\dots$ are mutually orthogonal,
$\langle (b-\epsilon)_+\rangle \leq_s \langle b_1\rangle \leq_s \langle b_2\rangle \cdots$, and  $\langle (b-\frac{1}{i})_+\rangle \leq_s \langle b_i\rangle $ for all $i$. Since
$\sup_i \langle (b-\frac{1}{i})_+\rangle =\langle b\rangle$ and $\langle b\rangle$ is full, we conclude from (i) that $\sum_{i=1}^\infty \langle b_i\rangle =\infty$. In particular,
$\langle b\rangle\leq \sum_{i=1}^\infty \langle b_i\rangle$. This implies that $(b-\epsilon)_+$ is Murray-von Neumann equivalent to an element $c$ in the hereditary subalgebra generated by $\sum_{i=1}^n b_i$. The elements  $(b-\epsilon)_+$ and $c$ are orthogonal, since $(b-\epsilon)_+$
is orthogonal to $\sum_{i=1}^n b_i$. Thus, $B$ is stable.

(ii)$\Rightarrow$ (i).
Say $y_i=\langle b_i\rangle $ for $i=1,2,\dots$, with $b_1,b_2,\dots$ mutually orthogonal and such that $b=\sum_{i=1}^\infty b_i$ is convergent.
Let us show that the hereditary subalgebra $B=\overline{b(A\otimes \mathcal K)b}$ has the property (S).
We have $\lambda(\sum_{i=1}^\infty y_i)=\infty$ for all non-zero $\lambda\in F(Cu(A))$ (see the proof of (iii)$\Rightarrow$(i)). Therefore,  $B$ has no non-zero bounded 2-quasitraces (a bounded 2-quasitrace on $B$ would extend to $A\otimes \mathcal K$ and give rise to a functional finite on $\sum_{i=1}^\infty y_i$). Suppose a quotient of $B$ is unital. Say, for simplicity, that $B$ is unital.
Since $\sum_{i=1}^\infty b_i$ is strictly
positive it is invertible, whence $\sum_{i=1}^n b_i$ is invertible for some $n$. This implies that  $b_{i}=0$ for all $i>n$ (since these elements are orthogonal to an invertible element), which  contradicts the fact that $\sup_i x_i$ is  full. Thus, $B$ has the property (S), and so by (ii) it is stable.
This implies that $y=\langle b\rangle=\infty$.

(iii) $\Rightarrow$ (i). Set $\sum_{i=1}^\infty y_i=y$. By \eqref{rinfty}, it suffices to verify that $\lambda(y)=\infty$ for all non-zero $\lambda\in F(Cu(A))$. We have $\lambda(y)\geq n\lambda(x_i)$ for all $i$ and all $n$. Taking supremum over $i$ we get $\lambda(y)\geq n\lambda(\sup_i x_i)$ for all $n$. Taking suprmeum over $n$ and using that $\sup_i x_i$ is full, we get
$\lambda(y)=\infty$ for all $\lambda\in F(Cu(A))$ non-zero.

Finally, suppose that in every quotient of  $A\otimes \mathcal K$ projections are finite and let us show that (ii)$\Rightarrow$(iii).
Let $y\in Cu(A)$ be such that $\lambda(y)=\infty$ for every non-zero $\lambda\in F(Cu(A))$.
 Say $y=\langle a\rangle $. Let us show that $\overline{a(A\otimes\mathcal K)a}$ is stable. It clearly has no
bounded 2-quasitraces. If a quotient of it is unital, then it would have a bounded 2-quasitrace since the unit would be stably finite. Thus, $\overline{a(A\otimes\mathcal K)a}$ has no unital quotients either.
It follows that it is stable. Hence, $y=\langle a\rangle =\infty$.
\end{proof}

\begin{rems}\label{rA_finite}{\rm
Notice that if $r_{A,a}<\infty$ for some full element $a\in A\otimes\mathcal K$ then $r_{A,\infty}=0$ and so we have (ii).
Theorem \ref{corona} (ii), in turn, implies the Corona Factorization Property for $A$.
Notice also that the condition that the projections in every quotient of $A\otimes \mathcal K$  be finite
implies that in order to verify the property (S) on a hereditary subalgebra of $A\otimes\mathcal K$ it suffices to show
that the algebra has no bounded 2-quasitraces. The lack of unital quotients follows automatically from this, as demonstrated
in the last paragraph of the proof of the preceding theorem.}
\end{rems}


If we seek to characterise in terms of the Cuntz semigroup the fact that every hereditary subalgebra $B$ of
$A \otimes \mathcal K$ with the property (S) is stable (without assuming that $B$ is full), then we
must have that (i) holds for all the closed two-sided ideals of $A\otimes \mathcal K$ that contain a
full element. This turns out to
be equivalent to the property of $\omega$-comparison in the Cuntz semigroup.
Recall from \cite{OPR} that
 $Cu(A)$ has the $\omega$-comparison property if $x\leq_s y_i$,
for $i=1,2,\dots$, implies $x\leq \sum_{i=1}^\infty y_i$.

The implication (i)$\Rightarrow$(ii) in the following corollary is the content of
\cite[Proposition 4.8]{OPR}.

\begin{cors}
The following propositions are equivalent.

(i) $Cu(A)$ has the $\omega$-comparison property.

(ii) If $B\subseteq A\otimes \mathcal K$ is $\sigma$-unital, hereditary,   and  has the property (S),  then it is stable.
\end{cors}

\begin{proof}
Let us show that $\omega$-comparison is equivalent to having (i) of the previous theorem for every
closed two-sided ideal of $A$ that contains a full element.

Suppose that we have $\omega$-comparison. Let $I$ be a closed two-sided ideal with a full element.
Let $(x_i)$ and $(y_i)$ be sequences in the ordered semigroup $Cu(I)$--which we view as an order ideal of $Cu(A)$--that satisfy (i) of the previous theorem.
By the $\omega$-comparison property, we have  for each $j$ that $x_j\leq \sum_{k=1}^\infty y_{i_k}$, where $(y_{i_k})$
is any infinite subsequence of $(y_i)$. It follows that $\infty\cdot x_j\leq \sum_{i=1}^\infty y_{i}$.
Taking supremum over $j$ we get that $\infty_I=\sum_{i=1}^\infty y_{i}$, where $\infty_I$ denotes
the largest element of $Cu(I)$.

Suppose that we have (i) of the previous theorem for  every closed two-sided ideal  $I$
of $A$ containing a full element. Let $x$, $(y_i)_{i=1}^\infty$ be elements in $Cu(A)$ such that $x\leq_s y_i$ for $i=1,2,\dots$.
We can project the elements $y_i$ to elements $y_i'\leq y_i$ such that $x\leq_s y_i'\leq \infty\cdot x$.
More specifically, let $x=\langle a\rangle $, $y_i=\langle b_i\rangle$, and let $I$ be the closed
two-sided ideal generated by $a$. Let $c\in (I\otimes \mathcal K)^+$ be strictly positive. Set
$y_i'=\langle cb_ic\rangle $. Then $x\leq_s y_i'\leq y_i$
and $y_i'\leq \infty\cdot x$. By (i) of the previous theorem, $\sum_{i=1}^\infty y_i'=\infty \cdot x$. Therefore,
$x  \leq \sum_{i=1}^\infty y_i'\leq \sum_{i=1}^\infty y_i$.
\end{proof}

\subsection{More on stability: closed two-sided ideals.}
The covering dimension of an open subset $U$ of a locally compact Hausdorff space $X$ is bounded above by the covering dimension of $X$.
A natural noncommutative generalization of this fact would be that the radius of comparison of an ideal $I$ of a C$^*$-algebra $A$ is bounded by the radius of comparison of $A$.
We shall see in Remark \ref{bad-ideal}, however, that this is not the case.
The basic problem is that the radius of comparison of the algebra $A$ is defined with respect to a full element of $A$, which is therefore not a member of any proper ideal.
Nonetheless, finite radius of comparison for an algebra tells us something about stability of ideals, as in the next result.

\begin{props}\label{ideals}
Let $A$ be such that  $r_{A,a}$ is finite for $a\in A_+$ strictly positive. Then  a closed 2-sided ideal
$I$ of $A$ has property (S)  if and only if
$\mathrm{M}_n(I)$ is stable for $n>r_{A,a}$.
 \end{props}

 \begin{proof}
It is clear that if $\mathrm{M}_n(I)$ is stable for some $n>r_{A,a}$ then
$I$ has property (S).
In order to prove the converse it is enough to consider the case that $r_{A,a}<1$
and show that if $I$ has property (S) then it is stable.
Suppose we are in this case. Let $e\in I_+$ be such that $fe=e$ and $gf=f$ for some $g,f\in I_+$.
By the Hjelmborg-R\o rdam criterion for stability ((i.e., \cite[Theorem 2.1 and Proposition 2.2]{hjelmborg-rordam})), in order to show that $I$ is stable it suffices
to find $x\in I$ such that $e=x^*x$ and $e$ is orthogonal to $xx^*$.

Notice that $g+(1-g)a(1-g)$ is a strictly positive element of $A$ and satisfies that $(g+(1-g)a(1-g))f=f$.
We replace $a$ by this element and assume that  $af=f$.
Let us show that $a-f$ is full. Let $J$  denote the closed two-sided ideal generated by $a-f$. The relation $af=f$ implies that $\pi(a)=\pi(f)$
is a projection, where $\pi$ is the quotient map onto $A/J$. Moreover, since $a$ is strictly positive
this projection must be the unit of $A/J$. Since $\pi(f)\in I/I\cap J$, we conclude that $I$ has
a unital quotient. This contradicts that $I$ has property (S). Thus, $a-f$ must be a full element of $A$.

Since $a-f$ is full, and $\langle f\rangle \ll \langle a\rangle $, we have $\langle f\rangle \leq N\langle a-f\rangle $ for some $N$. Consequently,
if $\lambda(\langle a-f\rangle )<\infty$ for some $\lambda\in F(Cu(A))$, then
\[
\lambda(\langle a\rangle )\leq \lambda(\langle a-f\rangle )+\lambda(\langle f\rangle )<\infty.
\]
Hence,
$\lambda$ is induced by a bounded 2-quasitrace on $A$. Since $I$ has the property (S), we must have
$\lambda(\langle f\rangle )=0$, and so $\lambda(\langle a\rangle )=\lambda(\langle a-f\rangle )$.
On the other hand, if
$\lambda(\langle a-f\rangle )=\infty$ then $\lambda(\langle a\rangle )=\lambda(\langle a-f\rangle )=\infty$. We conclude that
$\lambda(\langle a\rangle )=\lambda(\langle a-f\rangle )$ for all $\lambda\in F(Cu(A))$.
Thus,
\[ \lambda(\langle f\rangle )+\lambda(\langle a\rangle )=\lambda(\langle a-f\rangle )\,\hbox{ for all } \lambda\in F(Cu(A)). \]
Since $r_{A,a}<1$, we conclude that $f \precsim a-f$. Thus,
there is $x$ such that $e=x^*x$ and $xx^*\in \mathrm{Her}(a-f)$. Since
$a-f$ is orthogonal to $e$, we have that $e$ is orthogonal to $xx^*$.
This is the desired $x$.
 \end{proof}

In order to prove a sort of converse to Proposition \ref{ideals}, we derive some useful properties of stable C$^*$-algebras.
The first one is weaker than, though similar to, having stable rank one.
It is strong enough, however, to obtain that Murray-von Neumann equivalence of positive elements implies approximate unitary equivalence (Lemma \ref{mvn-aue}).

\begin{lms}\label{almost-sr1}
Let $A$ be a stable C$^*$-algebra.
Then every element of $A$ is the limit of invertible elements in the unitization $A^\sim$.
\end{lms}

\begin{proof}
Let $A = B \otimes \mathcal{K}$.
Then finite matrices over $B$ form a dense subset of $A$, so it suffices to show that every finite matrix over $B$ is the limit of invertible elements in $A^\sim$.

Let $x \in B \otimes \mathrm{M}_n$.
Let $x=ab$ for some $a,b \in B \otimes \mathrm{M}_n$ (this is possible using polar decomposition of $x$).
Viewing these in $B \otimes \mathrm{M}_n \otimes \mathrm{M}_2 \subset A$, we have
\[ \begin{pmatrix} x & 0 \\ 0 & 0 \end{pmatrix} = \begin{pmatrix} 0 & a \\ 0 & 0 \end{pmatrix} \begin{pmatrix} 0 & 0 \\ b & 0 \end{pmatrix}. \]
Since every nilpotent is the limit of invertible elements, we see that the right-hand side is the limit of invertibles, as required.
\end{proof}

\begin{lms}\label{mvn-aue}
Let $A$ be a stable C$^*$-algebra, $x \in A$.
Then there exists a sequence of unitaries $u_n \in A^\sim, n=1,2,\dots$ such that
\[ u_nxx^*u_n^* \to x^*x. \]
\end{lms}

\begin{proof}
By Lemma \ref{almost-sr1}, take a sequence of invertible elements $x_n \in A^\sim$ which converge to $x$.
Let $x_n = u_n|x_n|$
be the polar decomposition of $x_n$; since $x_n$ is invertible, $u_n$ is a unitary in $A^\sim$.
We have
\begin{align*}
\lim_{n\to\infty} u_n^*xx^*u_n &= \lim_{n\to\infty} u_n^*x_nx_n^*u_n^* \\
&= \lim_{n\to\infty} |x_n|^2 \\
&= \lim_{n\to\infty} |x_n|^*u_n^*u_n|x_n| \\
&=\lim_{n\to\infty} x_n^*x_n \\
&= x^*x.\qedhere
\end{align*}
\end{proof}

\begin{props}\label{stable-unitization}
Let $A$ be a C$^*$-algebra with the property (S). Then $r_{A^\sim,1}+1$ is the least number $n$
such that $\mathrm{M}_n(A)$ is stable  (if no such number exists we set $n=\infty$).
\end{props}

\begin{proof}
Let $n$ be the least number for which $\mathrm{M}_n(A)$ is stable.
From Proposition \ref{ideals} we obtain the bound \mbox {$n\leq r_{A^\sim,1}+1$}. Let us prove that
$r_{A^\sim,1}\leq n-1$.

Set $B:=A^{\sim}$.
Let $a,b \in (B \otimes \mathcal K)_+$ be positive elements satisfying
\begin{equation}
\lambda(\langle a\rangle ) +(n-1+\epsilon)\lambda(\langle 1\rangle ) \leq \lambda(\langle b\rangle ) \quad \text{for all }\lambda \in F(Cu(B)), \label{Ideal-Example-StrictIneq} \end{equation}
for some $\epsilon>0$. Express $a = a' + l$, $ b = b' + m$ where $a',b' \in A \otimes \mathcal K$ and
$l,m \in \mathcal K_+$. Then by letting $\lambda$ in \eqref{Ideal-Example-StrictIneq} be the normalised functional that vanishes on $A$, we have
\[ \mathrm{rank }(l) +n \leq  \mathrm{rank}(m). \]
By replacing $a$ and $b$ by Cuntz equivalent elements, we may suppose that $l = 1_{k-n}$ and $m = 1_{k} \oplus m'$ for some $k\in \N$ and $m'\in \mathcal K_+$.
Then $a \culeq 1_{k-n} \oplus |a'|$.
Moreover, since $\mathrm{M}_n(A)$ is stable, $|a'|$ is Murray-von Neumann equivalent to an element $a''$ of $\mathrm{M}_n(A)$, so that
$a \culeq 1_{k-n} \oplus a''$.
Let us show that $1_{k-n} \oplus a'' \culeq b$. We have
\[
b\cugeq 1_kb1_k = 1_k+1_kb'1_k \geq 1_k-b'' \geq 0,
\]
where $b''\in \mathrm{M}_k(A)_+$ denotes the negative part of $1_kb'1_k$.

Since $\mathrm{M}_n(A)$ is stable, $b''$ is Murray-von Neumann equivalent to some element $b''' \in \mathrm{M}_n(A)$.
By replacing $b'''$ and $a''$ by Murray-von Neumann equivalents, we may assume that $a''b''' = 0$ and that $\|a''\| \leq 1$.
In particular, this implies that
\[ 1_n-b''' \geq a''. \]

By \cite[Proposition 2.1]{Ror-stab}, $\mathrm{M}_k(A)$ is stable, and so Murray-von Neumann equivalent elements of $\mathrm{M}_k(A)$ are approximately unitarily
equivalent by Lemma \ref{mvn-aue}.
Thus, there are unitaries $u_m\in \mathrm{M}_k(A)^\sim$ such that $u_mb_-''u_m^*\to b'''$, and so
$u_m(1_k-b'')u_m^*\to 1_k-b'''$. We have
\begin{align*}
b \cugeq 1_k-b'' &\cueq 1_k-b''' \geq 1_{k-n}+a''\cugeq a.\qedhere
\end{align*}
\end{proof}

\begin{rems}\label{bad-ideal}{\rm
R\o rdam showed in \cite{Ror-stab} that for every natural number $n$, there exists a simple, stably finite algebra $A$ for which $\mathrm{M}_{n+1}(A)$ is stable but $\mathrm{M}_n(A)$ is not.
By Proposition \ref{stable-unitization} we see that for such an algebra $r_{A^\sim}=n$.
This shows that Theorem \ref{corona} cannot be improved to yield the Corona Factorization
Property for the ideals of $A$.}
\end{rems}

\subsection{When is $W(A)$ hereditary?}
Recall that $W(A)$ is the subsemigroup of $Cu(A)$ of elements $\langle a \rangle$
with $a\in \mathrm{M}_n(A)_+$ for some $n$.  In fact, $W(A)$ is the original definition of the Cuntz semigroup.  Here we consider the question of when this subsemigroup is hereditary, i.e., has the property that if $x \leq y$ in $Cu(A)$ and $y \in W(A)$, then $x \in W(A)$.  We prove that finite radius of comparison suffices. This result was previously unknown, even in the case of strict comparison.

\begin{thms}\label{hereditary}
Let $A$ be a C$^*$-algebra for which the projections in every quotient of $A\otimes \mathcal K$
are finite. Let $a\in A_+$ be strictly positive and suppose that
  $r_{A,a} < k\in \N$. If $\langle b \rangle \in Cu(A)$ is such that
\[\lambda(\langle b \rangle)\leq n\lambda(\langle a \rangle)\hbox{ for all }\lambda\in F(Cu(A)),\]
for some $n\in \mathbb{N}$, then  $b$ is Murray-von Neumann equivalent to an element of $\mathrm{M}_{2(n+k)}(A)_+$.
In particular,  $W(A)$ is a hereditary subset of $Cu(A)$.
\end{thms}

\noindent
Before proving Theorem \ref{hereditary}, we need a lemma.

\begin{lms}
If $b_1$ and $b_2$ are Murray-von Neumann equivalent to elements in $\mathrm{M}_n(A)$ then $b_1+b_2$ is Murray-von Neumann equivalent to an element  in $\mathrm{M}_{2n}(A)$.
\end{lms}

\begin{proof}
Upon identifying $\mathrm{M}_{2n}(A)$ with $\mathrm{M}_{2} \otimes \mathrm{M}_{n}(A)$ it is clear that there are orthogonal elements $\tilde{b}_{1}$ and $\tilde{b}_{2}$ in $\mathrm{M}_{2n}(A)_{+}$ which are Murray-von Neumann equivalent to $b_{1}$ and $b_{2}$, respectively.  Let the equivalence be implemented by $x_{1}$ and $x_{2}$ in $\mathrm{M}_{2n}(A)$, i.e., $x_{i}^{*}x_{i} = b_{i}$ and $x_{i}x_{i}^{*} = \tilde{b}_{i}$ for $i=1,2$. Then one has $b_{1}+b_{2} = (x_{1}+x_{2})^{*}(x_{1}+x_{2})$ and $(x_{1}+x_{2})(x_{1}+x_{2})^{*} \in \mathrm{M}_{2n}(A)$.
\end{proof}

\noindent
We can now prove Theorem \ref{hereditary}.

\begin{proof} Let $b$ and $a$ be as in the Theorem.
 Let us cover the interval $(0,\|b\|]$ by nonempty open intervals
$(I_i)_{i=1}^\infty$ such that  their centres form a strictly decreasing sequence (converging to $0$), and such that   $\overline{I_i}\cap \overline{I_{i+2}}=\emptyset$ for all $i=1,2,\dots$. Let $f,\tilde{f}\in \mathrm{C}_0((0,\|b\|])_{+}$
be positive functions supported on  $\bigcup_{i=1}^\infty I_{2i-1}$ and $\bigcup_{i=1}^\infty I_{2i}$
respectively, and such that $f(t)+\tilde{f}(t)=t$ for all $t\in (0,\|b\|]$. We then have that $f(b)+\tilde{f}(b)=b$. By the previous lemma,
it suffices to show that $f(b)$ and $\tilde{f}(b)$ are Murray-von Neumann equivalent to elements in  $\mathrm{M}_{n+k}(A)_+$.
We will prove this for $f(b)$;  the proof for $\tilde{f}(b)$ is similar.

Let us prove the existence of $z\in A\otimes \mathcal K$ such that $f(b)=z^*z$ and $zz^*\in \mathrm{M}_{n+k}(A)_+$.
Let us set $f_i := f \chii_{I_{2i-1}}$. Let $(I_{2i-1}')_{i=1}^\infty$
be a sequence of pairwise disjoint open intervals such that $\overline{I_{2i-1}}\subseteq I_{2i-1}'$ for all $i$.
For each $i$ we can find additional functions $g_i,h_i \in \mathrm{C}_0((0,\|b\|])_+$ of norm at most $1$ such that:
\begin{enumerate}
\item[(i)] $g_if_i = f_i$ for all $i$;
\item[(ii)] $h_ig_i = g_i$ for all $i$;
\item[(iii)] $\mathrm{supp}(g_i) \subset I_{2i-1}'$;
\item[(iv)] $\mathrm{supp}(h_i) = \bigcup_{j\geq i} I_{2j-1}'$.
\end{enumerate}
In particular, (iii) and (iv) imply that $h_ig_j = 0$ for $j < i$.

We shall find elements $z_i$ of $A\otimes \mathcal K$ for which $z_i^*z_i = f_i(b)$ and $z_iz_i^* \in \mathrm{M}_{n+k}(A)$.
Moreover, and crucially, we shall arrange that $z_i^*z_j = 0$ for $i \neq j$.
This will allow us to define $z$ by the convergent sum
\[
z = \sum_{i=1}^\infty z_i,
\]
and we see that $z^*z = f(b)$ while $zz^* \in \mathrm{M}_{n+k}(A)$.

From $\lambda(\langle b \rangle) \leq n\lambda(\langle a \rangle$) for all $\lambda$ and since  $r_{A,a}+\epsilon \leq k$ for some $\epsilon>0$, we get that
\begin{equation}
\lambda(\langle b \rangle) + (r_{A,a}+\epsilon)\lambda(\langle a \rangle) \leq (n+k)\lambda(\langle a \rangle)=\lambda(\langle a\otimes 1_{n+k} \rangle), \label{dtau-ineq}
\end{equation}
for all $\lambda\in F(Cu(A))$. In particular, $\langle h_1(b) \rangle \leq  \langle b \rangle \leq \langle a\otimes 1_{n+k} \rangle$ and so (by \cite[Proposition 2.4,
(i) $\Rightarrow$ (iv)]{Ror-uhf2}), there exists $s_1 \in A\otimes \mathcal K$ and $y_1 \in C^*(a\otimes 1_{n+k})_+$ such that
\[
s_1^*s_1 = g_1(b) \quad \text{and} \quad y_1(s_1s_1^*) = s_1s_1^*.
\]
Set $z_1 := s_1f_1(b)^{1/2}$, then $z_{1}^{*}z_{1} = f_{1}(b)$, and $z_{1}z_{1}^{*} \in \mathrm{M}_{n+k}(A)$. By a careful choice of $y_1$, there is no difficulty in supposing that it is strictly positive in
$C^*(a\otimes 1_{n+k})$.

The remaining $z_i$'s are found recursively, by finding for each $i$ elements $s_i \in A\otimes \mathcal K$ and $y_i \in \mathrm{M}_{n+k}(A)_+$
satisfying
\[
s_i^*s_i = g_i(b), \quad y_is_i = s_i, \quad y_iz_j = 0 \text{ for $j < i$},
\]
and
\[
\lambda(\langle h_i(b) \rangle) + (r_{A,a} +\epsilon) \lambda(\langle a \rangle) \leq \lambda(\langle y_i \rangle) \quad \text{for all $\lambda \in F(Cu(A))$},
\]
where $z_j = s_jf_j(b)^{1/2}$.

Having found $s_1, \dots, s_i$ and $y_1, \dots, y_i$, let us see how to obtain $s_{i+1}$ and $y_{i+1}$.
We will obtain $y_{i+1}$ by functional calculus on $y_i - s_is_i^*$.
We already know that $(y_i - s_is_i^*)z_j = 0$ for $j < i$.
But also,
\[
(y_i - s_is_i^*)z_i = (y_i - s_is_i^*)s_if_i(x)^{1/2} = s_if_i(b)^{1/2} - s_ig_i(b)f_i(b)^{1/2} = 0.
\]

For $\lambda\in F(Cu(A))$ we have
\begin{align*}
\lambda(\langle h_{i+1}(b) \rangle )+(r_{A,a}+\epsilon)\lambda(\langle a \rangle )+\lambda(\langle g_i(b) \rangle ) &\leq \lambda(\langle h_i(b) \rangle )+(r_{A,a}+\epsilon)\lambda(\langle a \rangle )\\
&\leq \lambda(\langle y_i \rangle ) \\
& \leq \lambda(\langle y_i-s_is_i^* \rangle )+\lambda(\langle g_i(b) \rangle ).
\end{align*}
If $\lambda(\langle g_i(b) \rangle )<\infty$ we conclude from here that
\begin{align}\label{hcomp}
\lambda(\langle h_{i+1}(b) \rangle )+(r_{A,a}+\epsilon)\lambda(\langle a \rangle )\leq \lambda(\langle y_i-s_is_i^* \rangle ).
\end{align}
Let us see that (\ref{hcomp}) also holds if $\lambda(\langle g_i(b) \rangle )=\infty$. Since $g_i(b)$
is an element of the Pedersen ideal of $A\otimes \mathcal K$, it suffices to show
that $y_i-s_is_i^*$ is  a full element of $A\otimes \mathcal K$, as this will imply $\lambda(\langle y_{i} - s_{i}s_{i}^{*} \rangle)$. Let $I$ be the closed two-sided
ideal generated by  $y_i-s_is_i^*$. Since $y_i$ is full and $y_is_i=s_i$, $y_i$ maps to a full projection in
$(A\otimes \mathcal K)/I$. Such a projection must be  finite. Hence, there is
$\tilde\lambda\in F(Cu(A \otimes \mathcal{K}/I))$ that is non-zero and densely finite. The functional $\tilde\lambda$ induces
a nonzero densely finite functional $\lambda\in F(Cu(A))$ that vanishes on $\langle y_i-s_is_i^* \rangle $. This
contradicts \eqref{hcomp}. Thus, $y_i-s_is_i^*$ is full.

We conclude from \eqref{hcomp} that $\langle h_{i+1}(b) \rangle  \leq \langle y_i - s_is_i^* \rangle $, and so (by \cite[Proposition 2.4, (i) $\Rightarrow$ (iv)]{Ror-uhf2}),
there exists $s_{i+1} \in A\otimes \mathcal K$ and $y_{i+1} \in C^*(y_i - s_is_i^*)$ such that
\[ s_{i+1}^*s_{i+1} = g_{i+1}(b) \quad \text{and} \quad y_{i+1}(s_{i+1}s_{i+1}^*) = s_{i+1}s_{i+1}^*. \]
We can arrange that $y_{i+1}$ be strictly positive in $C^*(y_i - s_is_i^*)$, so that all of the inductive hypotheses hold.
\end{proof}

Although at present there is no known example, it seems likely that the following question has a positive answer.

\begin{qus} Is there a C$^*$-algebra for which $W(A)$ is not a hereditary subset of $Cu(A)$?\end{qus}

\bibliography{radius}
\bibliographystyle{abbrv}

\end{document}